\documentclass[10pt]{amsart}
\usepackage{egastyle}

\usepackage{amssymb,amsmath,amscd,amsthm,amsfonts,wasysym,mathrsfs,amsxtra,stmaryrd,array,mathtools,url}

\usepackage[section]{placeins}
\usepackage{afterpage}
\usepackage{verbatim}

\input xy
\xyoption{all}

\newcommand{\Aut}{\mathrm{Aut}}

\newcommand{\Dc}{\mathcal{D}}

\newcommand{\arinj}{\ar@{^{(}->}}
\newcommand{\arsurj}{\ar@{->>}}
\newcommand{\areq}{\ar@{=}}

\newcommand{\wt}{\widetilde}

\newcommand{\bsm}{\begin{smallmatrix}}
\newcommand{\esm}{\end{smallmatrix}}

\makeatletter
\newtheorem*{rep@theorem}{\rep@title}
\newcommand{\newreptheorem}[2]{%
\newenvironment{rep#1}[1]{%
 \def\rep@title{#2 \ref{##1}}%
 \begin{rep@theorem}}%
 {\end{rep@theorem}}}
\makeatother

\newreptheorem{theorem}{Theorem}

\usepackage{scalerel,stackengine}
\stackMath
\newcommand\reallywidehat[1]{%
\savestack{\tmpbox}{\stretchto{%
  \scaleto{%
    \scalerel*[\widthof{\ensuremath{#1}}]{\kern-.6pt\bigwedge\kern-.6pt}%
    {\rule[-\textheight/2]{1ex}{\textheight}}%WIDTH-LIMITED BIG WEDGE
  }{\textheight}%
}{0.5ex}}%
\stackon[1pt]{#1}{\tmpbox}%
}
\begin{document}

\title[Bridgeland stability and Catalan numbers]{A note on Bridgeland stability conditions and Catalan numbers}

\author[Jason Lo]{Jason Lo}
\address{Department of Mathematics \\
California State University, Northridge\\
18111 Nordhoff Street\\
Northridge CA 91330 \\
USA}
\email{jason.lo@csun.edu}

\author[Karissa Wong]{Karissa Wong}
\address{Department of Mechanical Engineering\\
 University of California, Berkeley\\
 6141 Etcheverry Hall\\
  Berkeley CA 94720\\
   USA}
\email{karissawong@berkeley.edu}

\keywords{Bridgeland stability, Catalan numbers, autoequivalence}
\subjclass[2010]{Primary 14J27, 14F07; Secondary: 05E14, 05A15}

\begin{abstract}
In this short note, we describe a problem in algebraic geometry where the solution involves Catalan numbers.  More specifically, we consider the derived category of coherent sheaves on an elliptic surface, and the action of its autoequivalence group on its Bridgeland stability manifold.  In solving an equation involving this group action,  the generating function of Catalan numbers arises, allowing us to use asymptotic estimates of Catalan numbers to arrive at a bound for the  solution set.
\end{abstract}

\maketitle
\tableofcontents

\section{Introduction}

The ubiquity of Catalan numbers is well known and is documented in a great body of literature.  Stanley's book  \cite{stanley2015}, for example, provides a comprehensive survey on Catalan numbers, including 214 different scenarios in which Catalan numbers arise.

This article is motivated by  a problem involving a group action in algebraic geometry.  On the surface, this problem has nothing to do with Catalan numbers.  The solution to this problem  hinges on solving an equation of the form
 \[
   u = (A+Bu^2)w.
 \]
Here, $A$ and $B$ are real constants, and the goal is to solve for $u$ as a power series in $w$, and then compute the radius of convergence of the power series.  The radius of convergence indicates a `boundary' within which  our  problem from algebraic geometry has a solution.  That is, we have a primarily algebraic problem  for which the solution requires analytic techniques.

As it turns out, we can rewrite the above equation in a form that matches a quadratic equation giving rise to the generating function of Catalan numbers.  We then utilise classical asymptotic estimates of Catalan numbers from  combinatorics to  compute the radius of convergence for $u$ as a series in $w$, thus solving our problem.

In Section \ref{sec:context}, we give an informal overview of the context and motivation for our problem from algebraic geometry.  In doing so, we freely use the language of algebraic geometry.  As such, the technical details of Section \ref{sec:context} are intended for the reader who is familiar with the language of algebraic geometry.  The general reader should feel free to skim over Section   \ref{sec:context} just to garner a sense of the deeper side of the story, or to skip Section   \ref{sec:context} altogether and focus on  Section  \ref{sec:Catalan}, in which we describe the mechanics of   how Catalan numbers appear from the above equation.

\section{Context and Motivation}\label{sec:context}

\paragraph[Bridgeland stability and group actions] In algebraic geometry,  the notion of stability conditions on a triangulated  category $\Dc$  was formulated by Bridgeland to understand the stability of Dirichlet branes in string theory \cite{StabTC}.  The group $\Aut (\Dc)$ of autoequivalences - or `symmetries' - of the triangulated category $\Dc$ then acts on the space of Bridgeland stability conditions.  Studying this group action  helps us understand the geometry of the stability manifold itself \cite{SCK3}, which has deep implications in homological mirror symmetry, and helps  clarify  connections between algebraic geometry and other branches of mathematics, including dynamical systems \cite{dimitrov2013dynamical}.

Given any smooth projective surface $X$, the derived category $D^b(X)$ of coherent sheaves on $X$ is a triangulated category.  Arcara and Bertram showed that for any pair of divisors $(\omega, B)$ on $X$ where $\omega$ is ample, there is a corresponding Bridgeland stability condition $\sigma_{\omega,B}$ on $D^b(X)$ \cite{ABL}.  One can informally think of $(\omega,B)$ as a system of coordinates for a subset of the space of Bridgeland stability conditions on $D^b(X)$.

When $X$ is a Weierstra{\ss} elliptic surface in the sense of \cite{FMNT}, there is a natural non-standard autoequivalence of $D^b(X)$, namely the relative Fourier-Mukai transform $\Phi$ that comes from the moduli problem of parametrising rank-one, torsion-free sheaves on fibers of the elliptic fibration.  In \cite{Lo20}, it was shown that when $\omega$ is ``large enough'', the image of $\sigma_{\omega,B}$ under the action of $\Phi$ is again a stability condition of the form $\sigma_{\omega',B'}$ up to a $\widetilde{\mathrm{GL}}^+(2,\mathbb{R})$-action.  The key here is that  $\omega',B'$ can be explicitly computed in terms of $\omega, B$.  Also, the   group  $\widetilde{\mathrm{GL}}^+(2,\mathbb{R})$ is the universal covering space of the subgroup of $\mathrm{GL}(2,\mathbb{R})$ consisting of matrices with positive determinants, and its action on the space of stability conditions is much easier to describe than the action of $\Aut (D^b(X))$.  In other words, given $(\omega,B)$, we can solve the equation
\begin{equation}\label{eq1}
  \Phi \cdot \sigma_{\omega,B} = \sigma_{\omega',B'}\cdot g
\end{equation}
for some $(\omega',B')$ and some $g \in\widetilde{\mathrm{GL}}^+(2,\mathbb{R})$ as long as $\omega$ is ``large enough''.

\paragraph[From Bridgeland stability to power series] \label{para:stabtopower} A particular class of solutions to \eqref{eq1} can be obtained by writing the divisors $\omega,B, \omega',B'$ as $\mathbb{R}$-linear combinations of the fiber class $f$ and the section $\Theta$ that comes with a Weierstra{\ss} fibration.  When the coefficients of $f$ and $\Theta$ in $\omega,\omega'$ are chosen in a way that takes into account the elliptic fibration structure,  the existence of solutions to \eqref{eq1} reduces to the existence of solutions to the equation \cite[(10.4.1)]{Lo20}
\begin{equation}\label{eq2}
 u = ( (m+\alpha-e)-(m-\tfrac{e}{2})u^2)w
\end{equation}
in $u,w$, where $u,w$ are real variables that are assumed to be positive, and $m, \alpha \in \mathbb{R}_{>0}$ are constants that give the `direction' of the ray containing $\omega$ in the ample cone of $X$, while $e$ is the negative of the self-intersection number of $\Theta$.  The plane curve \eqref{eq2} first appeared in the first author's joint work  with Liu and Martinez \cite[(2.7.2)]{LLM}.  Through a change of variable, the condition that $\omega$ is ``large enough''  translates to $w$ being small enough.  By repeatedly applying \eqref{eq2}, we can write $u$ as a formal power series $f(w)$ in $w$ with real coefficients, and when $w$ is within the radius of convergence of the power series, we obtain a solution in $u,w$ to \eqref{eq2} and hence a solution to \eqref{eq1}.  The full details of the connection between solving \eqref{eq1} and solving \eqref{eq2} can be found in \cite[10.1-10.4]{Lo20}.

In short, if we can compute the radius of convergence of the  power series $f(w)$, then we would know how `large' $\omega$ needs to be in order for a solution to \eqref{eq1} to exist.  A natural approach to computing the radius of convergence of $f(w)$ is to  find a formula for its coefficients first, and then apply Hadamard's formula.  In computing the coefficients of $f(w)$, Catalan numbers appear.

\section{Power series and Catalan numbers}\label{sec:Catalan}

\paragraph[Catalan numbers] The \label{para:Catalan} Catalan numbers $c_n$ (where $n \geq 0$) satisfy the recurrence relation
\[
  c_n = \sum_{i+j=n-1}c_ic_j \text{\quad for $n \geq 1$}
\]
and $c_0=1$.  The sequence $\{c_n\}_{n\geq 0}$ starts off with $1, 1, 2, 5, 14, 42, 132, \cdots$.  We will need the following two well-known properties of the Catalan numbers:
\begin{itemize}
    \item[(a)] The generating function of the Catalan numbers $C=\sum_{n=0}^\infty c_nx^n$ solves the quadratic equation
\begin{equation}\label{eq7}
  C=1+xC^2.
\end{equation}
\item[(b)] We have a closed-form formula
\[
 c_n = \frac{1}{n+1}\binom{2n}{n}.
\]
\end{itemize}
For more details of (a) and (b) and their proofs, the reader may refer to  \cite[Example 4, 7.5]{10.5555/562056},  \cite[p.19]{roman2015introduction} or \cite[p.6-7]{flajolet_sedgewick_2009}, for instance.  In particular, given (a),  a standard approach to derive (b) is to apply the quadratic formula to \eqref{eq7} in conjunction with the binomial series $(1+x)^\alpha = \sum_{n=0}^\infty \binom{\alpha}{n} x^n$.

\paragraph[Power series solution] Let us write
\[
  A = m+\alpha-e, \text{\quad} B=-(m-\tfrac{e}{2})
\]
so that \eqref{eq2} can be written in the more compact form
\begin{equation}\label{eq3}
u=(A+Bu^2)w.
\end{equation}

\begin{lem}\label{lem1}
When $A \neq 0$, the power series $u$ that solves \eqref{eq3} is given by
\begin{equation}\label{eq5}
u = \sum_{n=0}^\infty c_n A^{n+1}B^nw^{2n+1}
\end{equation}
where $c_n$ is the $n$-th Catalan number.
\end{lem}

\begin{proof}
Let us set  $\wt{u}=\tfrac{u}{Aw}$ so that  \eqref{eq3} can be rewritten as
\begin{equation*}
\wt{u} = 1 + (ABw^2)\wt{u}^2.
\end{equation*}
From \ref{para:Catalan}(a), it follows that
\begin{equation*}
\wt{u}=\sum_{n=0}^n c_n(ABw^2)^n
\end{equation*}
where $c_n$ is the $n$-th Catalan number.  The claim \eqref{eq5}   follows immediately.
\end{proof}

In fact, the authors did not recognise the connection between equations of the form \eqref{eq3} and Catalan numbers at first.  It was only after some coefficients of $u$ were computed using a simple program \cite{CatalanCode}, and having those coefficients compared against the OEIS database \cite{OEISCatalan} that the connection became apparent.

\paragraph[Radius of convergence] Using Stirling's formula  $n! \thicksim \sqrt{2\pi n} (\frac{n}{e})^n$, we have the following asymptotic estimate of the Catalan number $c_n$
\begin{equation}\label{eq6}
c_n \thicksim \frac{4^n}{\sqrt{\pi} n^{3/2}}
\end{equation}
\cite[Theorem 3.1]{roman2015introduction}.

\begin{lem}\label{lem2}
When $A,B\neq 0$, the radius of convergence of the power series \eqref{eq5} is $ \frac{1}{2\sqrt{|AB|}}$.
\end{lem}

\begin{proof}
We first rewrite \eqref{eq5} as
\[
  u = w\left( \sum_{n=0}^\infty c_n A^{n+1}B^n (w^2)^n\right)
 \]
 and set
 \[
   \overline{u}(w) :=  \sum_{n=0}^\infty c_n A^{n+1}B^nw^n
 \]
 so that $u=w\overline{u}(w^2)$.  By \eqref{eq6} we have
 \[
  |c_nA^{n+1}B^n|^{1/n} \thicksim\frac{4}{(\sqrt{\pi})^{1/n}(n^{1/n})^{3/2}} |AB|\cdot |A|^{1/n}.
 \]
 Since $n^{1/n} \to 1$ as $n \to \infty$, it follows that
 \[
   \lim_{n \to \infty} |c_nA^{n+1}B^n|^{1/n} = 4|AB|,
 \]
 and so by Hadamard's formula \cite[2.4]{ahlfors}, the radius of convergence of $\overline{u}$ as a power series in $w$ is $1/(4|AB|)$.  Hence the radius of convergence of $u$ as a power series in $w$ is $\frac{1}{2\sqrt{|AB|}}$.
 \end{proof}

\paragraph[Back to Bridgeland stability] Combining the statements of Lemmas \ref{lem1} and \ref{lem2}, we have

\begin{prop}\label{prop1}
Assuming $m+\alpha-e$ and $m-\tfrac{e}{2}$ are both nonzero, the equation
\begin{equation*}
 u = ( (m+\alpha-e)-(m-\tfrac{e}{2})u^2)w
\end{equation*}
has a solution in $u$ as a convergent power series in $w$
\begin{equation*}
u = \sum_{n=0}^\infty c_n A^{n+1}B^nw^{2n+1}
\end{equation*}
with radius of convergence $\frac{1}{2\sqrt{|AB|}}$, where $c_n$ is the $n$-th Catalan number, $A=m+\alpha-e$ and $B=-(m-\tfrac{e}{2})$.
\end{prop}

Note that  the constants $m$ and $\alpha$ can always be chosen in \cite{Lo20} so that $A>0$ and $B<0$.  Also, in \cite[10.4]{Lo20}, the ample divisor $\omega$ mentioned in \ref{para:stabtopower} is written as $\beta\wt{\omega}$ for some fixed ample divisor $\wt{\omega}$, where $\beta$ is an analytic function in the variable $v:=1/w$ such that $\beta \thicksim v$ as $v \to \infty$.  When $w$ is within the radius of convergence $\frac{1}{2\sqrt{|AB|}}$, i.e.\ when $v > 2\sqrt{|AB|}$, the ample divisor $\omega$ becomes ``large enough'' and the equation \eqref{eq1} for Bridgeland stability conditions has a solution where $\omega',B'$ could be explicitly computed by following the recipe in \cite[10.1-10.4]{Lo20}.

\bigskip
\noindent
\textbf{Acknowledgments.} The authors would like to thank Mei Lim for valuable discussions.  The first author was partially supported by NSF Grant DMS-2100906.

\bibliography{refs}{}

\begin{thebibliography}{10}

\bibitem{ahlfors}
L.~V. Ahlfors.
\newblock {\em {Complex Analysis: an Introduction to the Theory of Analytic
  Functions of One Complex Variable}}.
\newblock {International Series in Pure and Applied Mathematics}. McGraw-Hill,
  New York, 3rd edition, 1953.

\bibitem{ABL}
D.~Arcara, A.~Bertram, and M.~Lieblich.
\newblock {Bridgeland-stable moduli spaces for K-trivial surfaces}.
\newblock {\em J. Eur. Math. Soc.}, 15(1):1--38, 2013.

\bibitem{FMNT}
C.~Bartocci, U.~Bruzzo, and D.~Hern\'{a}ndez-Ruip\'{e}rez.
\newblock {\em Fourier-Mukai and Nahm Transforms in Geometry and Mathematical
  Physics}, volume 276.
\newblock Birkh\"{a}user, 2009.
\newblock Progress in Mathematics.

\bibitem{StabTC}
T.~Bridgeland.
\newblock Stability conditions on triangulated categories.
\newblock {\em Ann. Math.}, 166:317--345, 2007.

\bibitem{SCK3}
T.~Bridgeland.
\newblock {Stability conditions on K3 surfaces}.
\newblock {\em Duke Math. J.}, 141:241--291, 2008.

\bibitem{dimitrov2013dynamical}
G.~Dimitrov, F.~Haiden, L.~Katzarkov, and M.~Kontsevich.
\newblock Dynamical systems and categories.
\newblock Preprint. arXiv:1307.8418 [math.CT], 2013.

\bibitem{flajolet_sedgewick_2009}
P.~Flajolet and R.~Sedgewick.
\newblock {\em Analytic Combinatorics}.
\newblock Cambridge University Press, 2009.

\bibitem{10.5555/562056}
R.~L. Graham, D.~E. Knuth, and O.~Patashnik.
\newblock {\em Concrete Mathematics: A Foundation for Computer Science}.
\newblock Addison-Wesley Longman Publishing, 2nd edition, 1994.

\bibitem{OEISCatalan}
OEIS~Foundation Inc.
\newblock {\em The On-Line Encyclopedia of Integer Sequences}, 2020.
\newblock https://oeis.org/A000108.

\bibitem{LLM}
W.~Liu, J.~Lo, and C.~Martinez.
\newblock {Fourier-Mukai transforms and stable sheaves on Weierstrass elliptic
  surfaces}.
\newblock Preprint. arXiv:1910.02477 [math.AG], 2019.

\bibitem{Lo20}
J.~Lo.
\newblock Weight functions, tilts, and stability conditions.
\newblock Preprint. arXiv:2007.06857 [math.AG], 2020.

\bibitem{roman2015introduction}
S.~Roman.
\newblock {\em An introduction to Catalan numbers}.
\newblock Springer, 2015.

\bibitem{stanley2015}
Richard~P. Stanley.
\newblock {\em Catalan Numbers}.
\newblock Cambridge University Press, 2015.

\bibitem{CatalanCode}
K.~Wong.
\newblock upowerseries.java.
\newblock
  https://github.com/KarissaWong/A-NOTE-ON-BRIDGELAND-STABILITY-CONDITIONS-AND-CATALAN-NUMBERS,
  2020.

\end{thebibliography}
\bibliographystyle{plain}

\end{document}